\documentclass[11pt]{amsart}
\usepackage{amsmath, amsthm, amssymb, amsfonts, amsxtra}
\usepackage{mathrsfs} 
\usepackage{graphics}
\usepackage{graphicx}
\usepackage{txfonts}
\usepackage[usenames]{color}
\usepackage{enumerate}
\usepackage{stmaryrd} 
                                                    
\numberwithin{equation}{section}
                                                 
\newtheorem{thm}{Theorem}[section]
\newtheorem*{*thm}{Theorem} 
\newcommand{\bt}{\begin{thm}}
\newcommand{\et}{\end{thm}}

\newtheorem*{*ex}{Example}

\newtheorem{cor}[thm]{Corollary}   
\newtheorem*{*cor}{Corollary}

\newcommand{\bc}{\begin{cor}}
\newcommand{\ec}{\end{cor}}

\newtheorem{lem}[thm]{Lemma} 
\newtheorem*{*lem}{Lemma}  

\newcommand{\bl}{\begin{lem}}
\newcommand{\el}{\end{lem}}

\newtheorem{prop}[thm]{Proposition}
\newcommand{\bp}{\begin{prop}}
\newcommand{\ep}{\end{prop}}
\newtheorem*{*prop}{Proposition} 

\newtheorem{defn}[thm]{Definition}
\newcommand{\bd}{\begin{defn}}   
\newcommand{\ed}{\end{defn}}
\newtheorem*{*defn}{Definition}

\newtheorem{rmrk}[thm]{Remark}  
\newtheorem*{*rmrk}{Remark} 

\newtheorem{claim}{Claim}

\newcommand{\be}{\begin{equation}}
\newcommand{\ee}{\end{equation}}

\providecommand{\R}{\mathbb{R}}

\providecommand{\N}{\mathbb{N}}

\newcommand{\E}{\mathbb{E}}
\providecommand{\sphere}{\mathbb{S}}

\providecommand{\mychi}{\raise2pt\hbox{$\chi$}}

\providecommand{\abs}[1]{\left\lvert#1\right\rvert}

\providecommand{\qgp}[2]{\raise3pt\hbox{$#1$}\!\big/\!\lower3pt\hbox{$#2$}}

\newcommand{\diam}{\operatorname{diam}}
\newcommand{\vare}{\varepsilon}

\newcommand{\Fm}{{\mathcal F}}

\newcommand{\set}{\rm{set}}

\newcommand{\disjointunion}{\sqcup}

\newcommand{\mass}{{\mathbf M}}

\newcommand{\vol}{\operatorname{Vol}}

\def\XXint#1#2#3{{\setbox0=\hbox{$#1{#2#3}{\int}$}
\vcenter{\hbox{$#2#3$}}\kern-.5\wd0}}

\begin{document}
\title[$\Fm$ limit with no geodesics]{ An intrinsic flat limit of
Riemannian manifolds with no geodesics}
\author[J. Basilio, D. Kazaras, \& C. Sormani]{Jorge Basilio, Demetre Kazaras, and Christina Sormani}
\address{CUNY Graduate Center}
\email{jorge.math.basilio@gmail.com}
\address{Simons Center for Geometry and Physics}
\email{demetre.kazaras@stonybrook.edu}
\address{CUNY Graduate Center and Lehman College}
\email{sormanic@gmail.com}
\thanks{The first author was partially supported by NSF DMS 1006059, the second author 
by NSF DMS 1547145, and the third author by NSF DMS 1612049}
\date{}
\begin{abstract}
In this paper we produce a sequence of Riemannian manifolds 
$M_j^m$, $m \ge 2$, which converge in the intrinsic flat sense
to the unit $m$-sphere with the restricted Euclidean distance.
This limit space has no geodesics achieving the distances between points,
exhibiting previously unknown behavior of intrinsic flat limits. 
In contrast, any compact Gromov-Hausdorff limit of a sequence of Riemannian manifolds is 
a geodesic space. Moreover, if $m\geq3$,
the manifolds $M_j^m$ may be chosen to have positive scalar curvature.
\end{abstract}
\maketitle
\thispagestyle{empty} 


\section{Introduction}\label{S:intro}
In 1981, Gromov introduced the Gromov-Hausdorff
distance between metric spaces, $d_{GH}((X_1, d_1), (X_2, d_2))$,
and proved that compact Gromov-Hausdorff limits
of Riemannian manifolds are
geodesic metric spaces \cite{Gromov-1981}.
That is, the distance between any pair of points in the limiting metric space
is realized by the length of a curve between the points.  
Recall that by the Hopf-Rinow Theorem, any compact Riemannian manifold $M$
with the length metric $d_M(p,q)$, defined by taking the infimum of
lengths of curves from $p$ to $q$, 
is a geodesic metric space.
If the Riemannian manifold is embedded into Euclidean space,  
one may also study a different metric space, $(M, d_\E)$
defined using the restricted Euclidean metric $d_\E$.   Observe that
the standard $m$-dimensional sphere $\sphere^m\subset\E^{m+1}$ with 
the Euclidian distance $d_\E$ is an example of a metric space 
$(\sphere^m, d_\E)$ with no geodesics.
Geodesic spaces and the Gromov-Hausdorff distance are reviewed
in Sections \ref{sec:backgroundgeo} and \ref{sec:backgroundghdist}.

In \cite{SW-JDG}, the third author and Wenger introduced the
intrinsic flat distance $d_{\mathcal{F}}$ between two integral current spaces
$(X_1,d_1,T_1)$ and $(X_2, d_2, T_2)$. An $m$-dimensional 
integral current space is a 
metric space $(X,d)$ with an 
collection of oriented biLipschitz charts and an
integral $m$-current $T$ such that $X$ is the set of
positive density of $T$. For example, a compact oriented
manifold $(M,g)$ has a canonical integral current space structure,
$\left(M, d_M, \int_M\right)$. If $M$ is embedded in
Euclidean space then we can create a different integral
current space, $\left(M, d_\E, \int_M\right)$ whose metric depends upon that embedding.  
We review these spaces and $d_\mathcal{F}$ in Sections \ref{sec:backgroundint} and \ref{sec:backgroundifdist}, respectively.

In the Appendix to \cite{SW-JDG}, the second author studied
a sequence of Riemannian $2$-spheres $(M_j, d_{M_j})$, $j=1,2,\ldots,$
constructed by removing a pair of tiny balls from two standard spheres and 
connecting the resulting boundaries by an increasingly thin tunnel.
These manifolds converge in the Gromov-Hausdorff
sense to a metric space, $(Y, d_Y)$, which is a pair of standard spheres 
connected by a line segment of length $1$ endowed with the length metric, $d_Y$.
On the other hand, the intrinsic flat limit of this sequence
consists of the pair of spheres $Y'\subset Y$ {\em{without}} the line segment.  
See Figure~\ref{fig-not-length}.
\begin{figure}[h!]
\includegraphics[scale=.4]{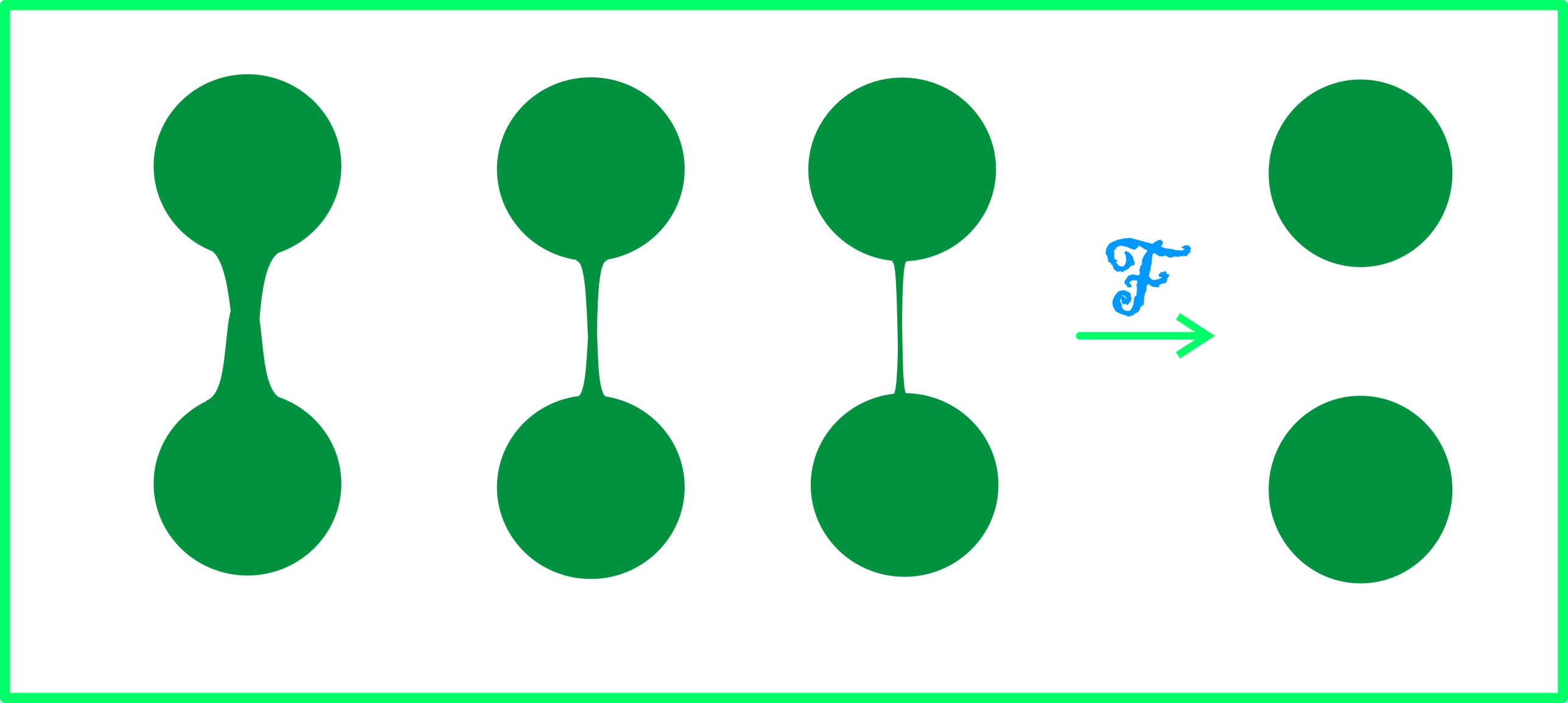}
\caption{The tunnel
disappears under $\mathcal{F}$ convergence so that
the $\mathcal{F}$ limit space is not a geodesic space \cite{SW-JDG}.}
\label{fig-not-length}
\end{figure}

\noindent In other words,
\[
\left(M_j,d_{M_j},\int_{M_j}\right)\to\left(Y',\,\, d_Y|_{Y'\times Y'}, \,\,\int_{Y'}\,\right)
\]
with respect to the intrinsic flat distance.
The line segment does not appear in the intrinsic flat limit 
since the current $\int_{Y'}$ 
does not have positive density there. Since $Y'$ is not path-connected, this example shows that 
intrinsic flat limits of Riemannian manifolds need not be geodesic metric spaces.   

When the second author presented this paper at the {\em Geometry Festival}
in 2009, various mathematicians asked whether it was possible to prove that
intrinsic flat limits of Riemannian manifolds are locally geodesic spaces:
{\em for all $p\in M_\infty$, is there a neighborhood
$U$ about $p$ such that all pairs of points in $U$ are joined by a geodesic
segment?} In this paper we provide an example demonstrating that this is not the case. 
More precisely, we prove the following theorem.
\bt\label{T:mainresult}
For each $m\geq2$, there is a sequence of closed, oriented 
$m$-dimensional Riemannian manifolds $M_j^m$,
so that the corrisponding integral current spaces 
$\mathcal{F}$-converge to 
\be
M_\infty = \left(\sphere^m,\,\, d_{\E^{m+1}},\,\, \int_{\sphere^m}\,\right). 
\ee
If $m\geq3$, the manifolds $M_j$ may be chosen to have positive scalar curvature.
\et
Since the metric space $(\sphere^m,\,\, d_{\E^{m+1}})$ is not a length 
space, we immediately obtain a negative answer to the above question. 
\begin{cor}\label{cor:maincor}
For $m\geq2$, the intrinsic flat limit of closed, oriented 
$m$-dimensional Riemannian manifolds need not be locally geodesically complete.
\end{cor}
To construct the Riemannian manifolds $M_j$ in Theorem \ref{T:mainresult},
we modify the standard $m$-sphere by gluing in increasingly thin tunnels
between pairs in an increasingly dense collection of removed balls.
The lengths of the tunnels approximate the Euclidean distance
between the points at the centers of the balls they replace.
See Figure~\ref{fig-many-tunnels}. The original plan for this construction in 
two dimensions was conceived by the first and last authors and presented a few
years ago but never published.

\begin{figure}[h!]
\includegraphics[scale=.4]{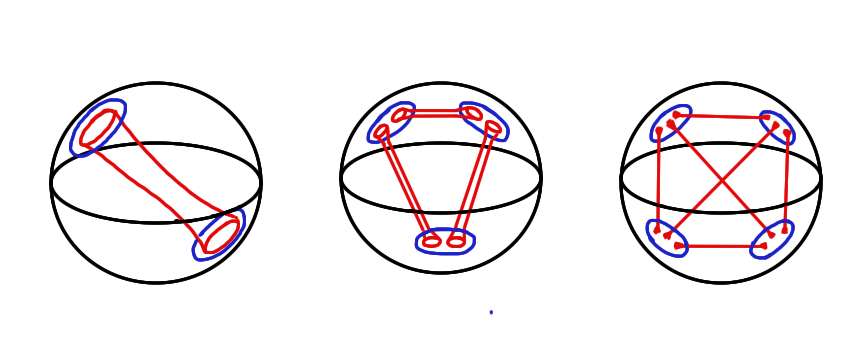}
\caption{A schematic depiction of the sequence in Theorem \ref{T:mainresult}}
\label{fig-many-tunnels}
\end{figure}

In \cite{Gromov-Plateau}, Gromov suggested that perhaps the natural notion of 
convergence for sequences of manifolds with positive scalar curvature is the intrinsic 
flat convergence. Note that the classic Ilmanen Example \cite{SW-JDG}
demonstrates that a sequence of three dimensional manifolds of positive scalar 
curvature with increasingly many increasingly thin increasingly dense wells has 
no Gromov-Hausdorff limit but does converge naturally in the intrinsic flat sense 
to a sphere. Thus the second and third authors decided to complete the 
construction of our example with positive scalar curvature when the dimension 
is $\ge 3$.   

In order to achieve the positive scalar curvature condition, we
design the tunnels using a refinement of the classical Gromov-Lawson 
construction \cite{GL80} recently obtained by Dodziuk \cite{Dod18}. 
See also Schoen-Yau's classic tunnel construction in
\cite{Schoen-Yau-tunnels}. Topologically,
each $M_j$ is the connected sum of a large number of copies of $S^1\times S^2$.

We begin the construction in Section \ref{sec:revisedpipe} by 
proving our main technical result, Theorem~\ref{T:pipefilling}, which 
allows us to estimate the intrinsic flat distance between a space with a tunnel and a space
with a thread.
This is an extension of the pipe-filling technique developed by the second author 
in the Appendix of \cite{SW-JDG}.
In \cite{SW-JDG}, this technique was applied to prove the $M_j'$ 
(the pair of spheres with a tunnel between them) and
$Y$ (the pair of spheres with a thread between them)
could be isometrically embedded into a common space.
This then allowed one to estimate the
intrinsic flat distance and Gromov-Hausdorff distance between
$M_j'$ and $M_\infty'\subset Y$.   
In  Theorem~\ref{T:pipefilling}, we start with an arbitrary Riemannian
manifold and a pair of points in that manifold and a length $L$ less
than the distance between the points.   
We show that the manifold
created by joining the points with a thin enough tunnel of length $L$
is close in the intrinsic flat sense to the space 
created by joining the points with a thread of length $L$  [Theorem~\ref{T:pipefilling}].   
As before, our estimate on 
the intrinsic flat distance between the space depends on volumes
and diameters while the Gromov-Hausdorff distance depends on the 
width of the tunnel and diameters.

We apply this pipe filling technique inductively
in Proposition~\ref{P:Fclosetospherewthreads} to prove that
our sequence, $M_{\vare_j}$, of spheres with increasingly dense tunnels 
[Definitions~\ref{D:defofXepsilon} and~\ref{D:ICSassociatedtoXep}] are
increasingly close to a sequence, $N_{\vare_j}$, of spheres with increasingly 
dense threads [Definitions~\ref{D:defofYepsilon} and~\ref{D:ICSassociatedtoYep}].
We apply a direct filling construction to prove that the spheres with
the increasingly dense threads viewed as integral current spaces,
$N_{\vare_j}$ converge to the sphere with the restricted Euclidean metric,
$M_\infty=(\sphere^m, d_{\E}, \int_{\sphere^m})$, in the intrinsic flat sense 
[Proposition~\ref{P:Fclosetospherewrestricted}]. Note that the threads are not
part of the integral current space, $N_{\vare_j}$, as they have lower density.
These are combined to prove Theorem~\ref{T:mainresult}.

It should be noted that the manifolds in our sequence have closed minimal 
hypersurfaces of increasingly small area. These stable minimal surfaces are 
located within the increasingly thin tunnels. In \cite{Sormani-Scalar}, the 
third author has conjectured that if one were to impose a uniform positive lower bound on 
\be
MinA(M^3)= \min\{ Area(\Sigma^2):\,\, \Sigma \textrm{ closed minimal in } M^3\}
\ee
that this combined with nonnegative scalar curvature would guarantee that the 
limit space has geodesics between every pair of points. This is a three dimensional conjecture.
It would be interesting to explore if other kinds of tunnels might be constructed in
higher dimensions to create limits with no geodesics even with a uniform 
positive lower bound on MinA.

The authors would like to thank Brian Allen, Edward Bryden, Lisandra Hernandez, 
Jeff Jauregui, Sajjad Lakzian, Dan Lee, Raquel Perales, and
Jim Portegies for many conversations about scalar curvature and convergence.  We 
would like to thank Jozef Dodziuk, Misha Gromov, Marcus Khuri, Blaine Lawson, 
Rick Schoen, and Shing-Tung Yau for their interest in these questions.

\section{Background}\label{sec:background}
Here we provide a brief review of metric spaces, length spaces and
integral current spaces, followed by a review of the Gromov-Hausdorff and
intrinsic flat distances between these spaces.
\subsection{Review of Geodesic Spaces:}\label{sec:backgroundgeo}
Given any metric space, $(X,d)$, 
one may define the length of a rectifiable curve $C:[0,1]\to X$ as
follows:
\be\label{Back1}
L_d(C)= \sup \left\{ \sum_{i=1}^N d(C(t_i), C(t_{i-1})): \, 
	0=t_0<t_1<\cdots <t_N=1, \, N \in \N\right\}.
\ee
If every pair of points in $X$ is connected by a rectifiable curve,
then one can define the induced length metric, $d_X$, on $X$ as
follows:
\be \label{Back2}
d_X(p,q) = \inf\{ L_d(C)\colon C(0)=p, \, C(1)=q \}. 
\ee
If $X$ is compact, then this infimum is achieved and any
curve achieving the infimum is called a {\emph{minimizing geodesic}}.   
A geodesic metric space, is a metric space of the form, $(X, d_X)$,
in which the distance between any pair of points is achieved
by a minimizing geodesic.  See \cite{BBI} for more details about length structures.

Given a pair of geodesic spaces, $(X_i, d_{X_i})$ with base points $x_i\in X_i$ for
$i=1,2$, one may produce a new geodesic space by identifying $x_1$ 
with $x_2$. We denote this by
\be
X_1 \disjointunion_{x_1 \sim x_2} X_2.
\ee
The metric on $X_1 \disjointunion_{x_1 \sim x_2} X_2$ is defined using 
curves that pass from $X_1$ to $X_2$ via the identified points as in (\ref{Back1}) and
the new distance is defined as in (\ref{Back2}).   

Consider the example of the sphere with the restricted
Euclidean metric $(\sphere^m, d_\E)$.  By the Law of Cosines:
\be\label{E:lawofcosines}
d_\E(x,y) = \sqrt{2 - 2\cos d_{\sphere^m}(x,y)} 
\ee
This is biLipschitz equivalent to $d_{\sphere^m}$ with Lipschitz
constants $1$ and $\pi/2$. However, $(\sphere^m, d_\E)$ is
not a geodesic space. In fact, for any pair of distinct points $x,y\in\sphere^m$, 
there is no geodesic connecting $x$ and $y$. Indeed,
there cannot exist midpoints, $z$, such that
\be
d_\E(x, z)= d_\E(y, z) = d_\E(x,y)/2.
\ee
\subsection{Review of Uniform Convergence of a metric space}
Fix a space $X$. A sequence of metrics $(d_i)$, $i=1,2,\ldots$, on $X$ is said to 
{\emph{converge uniformly}} to a metric $d$ on $X$ if
if 
\be
	\sup_{x,x' \in X} \abs{d_i(x,x') - d(x,x')} \to 0 
\ee
as $i\to\infty$.
A sequence of metric spaces $(X_i,\rho_i)$ converges uniformly to 
a metric space $(X,d)$ if there exists a sequence of metrics $(d_i)$ on 
$X$ such that $(X_i,\rho_i)$ is isometric to $(X,d_i)$, for all $i$, and $d_i$ 
converges uniformly to $d$. 
\subsection{Review of the Gromov-Hausdorff Distance}\label{sec:backgroundghdist}
The Gromov-Hausdorff distance was defined in \cite{Gromov-1981}.  See also
\cite{BBI}.

A distance preserving map, $F: (X_1, d_1) \to (X_2, d_2)$ is a map such that
\be
d_2(F(a), F(b)) = d_1(a,b) \qquad \forall a,b, \in X_1.
\ee
Observe that the standard embedding $F: \sphere^2 \to \E^3$ is distance 
preserving when viewed as $F: (\sphere^2, d_\E) \to (\E^3, d_\E)$ but is not 
distance preserving when $\sphere^2$ is given the length metric $d_{\sphere^2}$.

The Gromov-Hausdorff distance between two metric spaces is defined as
\be
d_{GH}((X_1, d_1), (X_2, d_2)) =\inf_Z d^Z_H(\varphi_1(X_1), \varphi_2(X_2))
\ee
where the infimum is taken over all metric spaces $(Z, d^Z)$ and over all
distance preserving maps $\varphi_i: (X_i, d_i) \to (Z, d^Z)$.
Here the Hausdorff distance between subsets $A,B \subset Z$ is defined by
\be
d^Z_H(A, B) := \inf\{ r:\, A\subset T_r(B), \, B\subset T_r(A)\}
\ee
where $T_r(U):=\{x \in Z\colon d^Z(x,U)<r\}$.

Clearly, one may estimate the Hausdorff distance $d_{GH}((X_1, d_1), (X_2, d_2))$ from above
by constructing a common space $(Z,d^Z)$ with distance preserving embeddings of $X_1$
and $X_2$. For instance, we may estimate the distance to a single point space, $(\{p\}, 0)$:
\be
d_{GH}((\sphere^2, d_\E), (\{p\},0) ) \le d_{H}^{\E^3}(\varphi(\sphere^2), \bar{0}) \le 1
\ee
taking $\varphi: \sphere^2 \to \E^3$ to be the standard embedding. However, 
$(\sphere^2, d_{\sphere^2})$ does not have a distance preserving map into $\E^3$.  

Given a metric space, $(X, d)$, let $N_X(r)$ be the maximum number
of disjoint balls of radius $r$ in $X$.    
Gromov's Compactness Theorem states that any sequence of
metric spaces, $X_j$, with a uniform upper bound, $N_{X_j}(r)\le N(r)$,
for all $r$, has a subsequence $X_{j_i}$ which converges in the
Gromov-Hausdorff sense to a compact metric space.  He proved the
converse as well: if compact $X_j$ converges to a compact $X$ in the Gromov-Hausdorff sense, 
then there is uniform upper bound, $N_{X_j}(r)\le N(r)$ \cite{Gromov-1981}. 
As a final remark, if a metric space $X$ is the Gromov-Hausdorff
limit of geodesic spaces, then $X$ is a geodesic space as well (\cite{Gromov-1981}).  
\subsection{Review of Integral Current Spaces}\label{sec:backgroundint}
We now provide an intuitive description of integral current spaces
which were first defined in \cite{SW-JDG} based upon work of
Ambrosio-Kirchheim in \cite{AK}.

An {\emph{$m$-dimensional integral current space}} $(X,d,T)$ is a metric 
space $(X,d)$ equipped with an $m$-dimensional integral current $T$. For the 
special case where $X$ is an oriented smooth $n$-dimensional manifold and $d$ is
a Lipschitz distance function on $X$, then $(X,d)$ has a cannonical $m$-dimensional 
integral current space structure given by 
\be
T(\omega) = \int_M \omega
\ee
for any $m$-form $\omega$ on $X$. For instance, 
$\left(\sphere^m, d_{\sphere^m}, \int_{\sphere}\right)$ and 
$\left(\sphere^m, d_{\E^{m+1}}, \int_{\sphere}\right)$ are both integral current spaces.

When $X$ is only a metric space, however, there is no general notion of a smooth differential
form. To define an integral current structure on such an $X$, one applies the methods 
of DiGeorgi \cite{DeGiorgi} and Ambrosio-Kirchheim \cite{AK} to replace the role of 
smooth $m$-forms $\omega$ by $(m+1)$-tuples $(\pi_0, \pi_1,...,\pi_m)$ of 
Lipschitz functions $\pi_i:X\to\mathbb{R}$. The notion of an integral current acting on 
such tuples is developed in detail in \cite{AK}. So long as $X$ is covered almost 
everywhere by a countable collection of biLipschitz charts $\varphi_i: A_i \subset \R^m \to X$,
with disjoint images, then 
\be \label{Back3}
T(\pi_0, \pi_1,...,\pi_m):= \sum_{i=1}^\infty f\circ \varphi_i d(\pi_1\circ\varphi_i)
\wedge \cdots \wedge d(\pi_m\circ\varphi_i)
\ee
defines an integral current on $X$.

The mass, mass measure, and density of an integral current space are defined in 
\cite{SW-JDG}, based upon \cite{AK}. For the purposes of this paper, we will only 
consider cases where $X$ is an oriented $m$-dimensional Riemannian manifold 
or is glued together from such manifolds. For these cases, the mass of $X$ is 
$\mathcal{H}^m(X)$, the mass measure is the $m$-dimensional Hausdorff
measure $\mathcal{H}^m$, and the density is 
\be
\Theta(p) =\liminf_{r\to 0} \frac{ \mathcal{H}^m(B_p(r))}{\omega_m r^m}.
\ee
Since integral current spaces $(X,d,T)$
are defined in \cite{SW-JDG} so that 
\be
X= \set(T):=\{ x\in \bar{X}: \, \Theta(x)>0\},
\ee
we see that only top dimensional regions in spaces created by gluing
together manifolds form part of $X$.   
\subsection{The Intrinsic Flat Distance:}\label{sec:backgroundifdist}
The intrinsic flat distance between two integral current spaces
was defined by the second author and Wenger in \cite{SW-JDG}
imitating Gromov's definition of the Gromov-Hausdorff distance:
\be
d_{\mathcal{F}}\big((X_1, d_1, T_1), (X_2, d_2, T_2)\big) =\inf
d_F^Z\big(\varphi_{1\#}(T_1), \varphi_{2\#}(T_2)\big)
\ee 
where the infimum of the flat distance, $d_F^Z$,
is taken over all integral current spaces $Z$
and all distance preserving maps $\varphi_i:X_i \to Z$.
Here $\varphi_\#(T)$ is the pushforward of the current
structure to a current on $Z$ defined by
\be
\varphi_\#T(\omega)=T(\varphi^*\omega)
\ee
where
\be
\varphi^*(\pi_0, \pi_1, ..., \pi_m) =(\pi_0\circ \varphi, ..., \pi_m\circ\varphi).
\ee
Recall the flat distance between two currents on an integral current space $Z$, is
\be
d_F^Z(\varphi_{1\#}(T_1), \varphi_{2\#}(T_2))
= \inf\{ \mass(A) + \mass(B):\, A +\partial B = \varphi_{1\#}(T_1)- \varphi_{2\#}(T_2) \}
\ee
where $A$ is an $m$-dimensional integral current on $Z$ and $B$
is $(m+1)$-dimensional.

To estimate the intrinsic flat distance between two $m$-dimensional 
integral current spaces, one constructs a metric
space $Z$ which is glued together from pieces which are 
biLipschitz $(m+1)$-dimensional manifolds $B_i$, such that
\be
B= \bigcup B_i \textrm{ where }
\bigcup_i \partial B_i = \varphi_1(X_1) \cup \varphi(X_2) \cup A.
\ee
Here $A$ consists of the leftover pieces of boundary. If this is done in an oriented way so that 
\be\label{E:bg-IF-estimate-1}
\int_{\varphi_1(X_1)}\omega
-\int_{ \varphi_2(X_2)}\omega
=\int_{\partial B}\omega + \int_A\omega=
\int_B d\omega + \int_A \omega, 
\ee
then, since $B$ and $A$ have weight one, we have
\be\label{E:bg-IF-estimate-2}
d_{\mathcal{F}}((X_1, d_1, T_1), (X_2, d_2, T_2)) \le \mathcal{H}^n(A)
+\mathcal{H}^{n+1}(B).
\ee

For example, if we wish to estimate the intrinsic flat distance between 
$\left(\sphere^2, d_{\sphere^2},\int_{\sphere^2}\right)$ and 
$\left(\sphere^2, d_{\E},\int_{\sphere^2}\right)$,
then we can consider the metric space $Z'$ defined in (\ref{Back7}) to see that
\be
d_{\mathcal{F}}\left(\left(\sphere^2, d_{\sphere^2},\int_{\sphere^2}\right),
\left(\sphere^2, d_{\E},\int_{\sphere^2}\right)\right) \le 
\mathcal{H}^3(D^3)+ \mathcal{H}^3(\sphere^3_+) = \pi + 2\pi.
\ee
As in the Gromov-Hausdorff setting, it may be difficult to construct integral current spaces $Z$ 
with distance preserving maps $\varphi_i: (X_i, d_i) \to Z$, $i=1,2$, that give
a sharp estimate on the distance between the spaces being considered.

The intrinsic flat distance between two integral current spaces is $0$ if and only if
there is a current preserving isometry between them, see \cite{SW-JDG}. Thus, for example, 
\be
d_{\mathcal{F}}\left(\left(\sphere^m, d_{\sphere^m}, \int_{\sphere^m}\right),
\left(\sphere^m, d_{\E}, \int_{\sphere^m}\right)\right)>0.
\ee 

For a sequence of integral current spaces, the Gromov-Hausdorff and Intrinsic 
Flat limits do not necessarilly agree. However, if a sequence
of integral current spaces converges in the Gromov-Hausdorff sense
then a subsequence converges in the intrinsic flat sense to a subset
of the Gromov-Hausdorff limit, see \cite{SW-CVPDE}. For example,
the pair of spheres joined by increasingly thin tunnels converges
in the Gromov-Hausdorff sense to a pair of spheres joined by a thread,
$Y$ of \cite{SW-JDG}. The same sequence converges in the intrinsic flat sense to the
pair of spheres with the thread removed $M'_\infty$.
\subsection{The Gromov-Lawson construction}
A key ingredient in the proof of Theorem \ref{T:mainresult}
is the construction of particular Riemannian metrics with positive scalar curvature on the
cylinder $S^{m-1}\times[0,1]$ for $m\geq3$.
These metrics smoothly transition between a small round metric on $S^{m-1}\times\{1\}$
and the boundary of a ball in an arbitrary Riemanian 
manifold of positive scalar curvature on $S^{m-1}\times\{0\}$. 
Gluing two such cylinders together along the round boundary components,
one obtains a long and thin tunnel -- called {\emph{Gromov-Lawson}} tunnels -- 
transitioning between two balls in positive scalar curvature manifolds.
This allows one, for instance, to produce metrics of positive scalar curvature on the 
connected sum of manifolds with positive scalar curvature. 
Although the original construction is due to Gromov-Lawson \cite{GL80},
we will require a recent generalization due to Dodziuk \cite{Dod18}. 

For our constructions in Sections \ref{sec:revisedpipe} and \ref{sec:construction},
we will require a detailed description of the construction when applied
to manifolds which have regions isometric to geodesic balls in the standard sphere. 
For such an $m$-dimensional positive scalar curvature manifold $(M^m,g_M)$,
fix two points $p,q\in M$ and choose a radius $\rho>0$ so that 
the geodesic balls $B^M_\rho(p)$ and $B^M_\rho(q)$
are disjoint and isometric to geodesic balls in the standard sphere. 
Topologically, one can perform surgery on the points $\{p,q\}$ by removing the balls
and attaching the cylinder $U:=\sphere^{m-1}\times[0,1]$
to the resulting geodesic sphere to obtain 
\[
M\#(\sphere^{m-1}\times \sphere^1)=
\left(M\setminus (B_\rho(p)\sqcup B_\rho(q))\right)\cup U.
\]
The Gromov-Lawson construction provides a Riemannian metric on $U$
which smoothly attaches to $g_M$ on $M\setminus (B_\rho(p)\sqcup B_\rho(q))$,
producing a {\emph{positive scalar curvature metric}} on 
$M\#(\sphere^{m-1}\times\sphere^1)$. In \cite{Dod18}, 
a refinement of the construction is given in order to obtain a qualitative description
of the geometry of the metric on $U$.
\bt\label{T:GLtunnel}\cite[Proposition 1]{Dod18}
Let $p,q\in (M^m,g_M)$ and $\rho>0$ be as above, and let $L>0$. There is a 
number $\rho_0\in(0,\rho)$ and a Riemannian metric $\bar g=\bar g_{\rho_0,L}$ on 
$U=U_{\rho_0,L}$ satisfying the following:
\begin{enumerate}
\item $\bar g$ has positive scalar curvature;
\item $\bar{g}$ and $g_M$ can be glued together to form a smooth metric on 
\[
M\setminus(B_p(\rho)\sqcup B_q(\rho))\cup U;
\]
\item $\mathrm{dist}_H(\sphere^{m-1}\times\{0\},\sphere^{m-1}\times\{1\})=L$ and $\mathrm{diam}(U)=\mathcal{O}(L)$;
\item $\vol(U)=\mathcal{O}(L\rho_0^{m-1})$;
\item if $\gamma\subset U$ is a continuous curve connecting the two boundary 
components of $U$, then the tubular neighborhood of $\gamma$ of radius
$2\pi \rho$ contains $U$ i.e.
\[
U\subset\{x\in U\colon\mathrm{dist}_{\bar{g}}(x,\gamma)\leq2\pi\rho\}.
\]
\end{enumerate}
Moreover, $\rho_0$ can be chosen arbitrarily small and the estimates in items $(3)$ and $(4)$
depend only on the geometry of $B_p(\rho)$ and $B_q(\rho)$.
\et
In \cite{Dod18}, Theorem \ref{T:GLtunnel} is proven by an explicit construction 
of a Riemannian metric $\bar g$ satisfying conditions $(1)$ through $(5)$ above. 
We give a qualitative description of this metric.

\begin{rmrk}\label{rmk:gromovlawson}
For given parameters $\rho_0\in(0,\rho)$ and $L>0$, there exists
a number $L'$ and a function $r_{\rho_0,L}:[\frac{-L'}{2},\frac{L'}{2}]\to(0,\rho]$
so that
\be
U_{\rho_0,L}=\{(t,x)\in{\Big{[}}\frac{-L'}{2},\frac{L'}{2}{\Big{]}}\times\mathbb{R}^m
	\colon |x|=r(t)\}\subset\mathbb{R}^{m+1}
\ee
and the metric $\bar g_{\rho_0,L}$ in Theorem \ref{T:GLtunnel} is the restriction of the flat metric
on $[\frac{-L'}{2},\frac{L'}{2}]\times\mathbb{R}^m$.
Notice that $L'<L$, but the length of the graph of $r_{\rho_0,L}$ is $L$.

For clarity, there are unit spheres on the $t$-axis equipped with geodesic balls,
$B_{\rho_0}^-\subset \partial B^{\mathbb{R}^{m+1}}_{-y_0}(1)$ and 
$B_{\rho_0}^+\subset \partial B^{\mathbb{R}^{m+1}}_{y_0}(1)$, so that 
\be
\left(\partial B^{\mathbb{R}^{n+1}}_{-y_0}(1)\setminus B_{\rho_0}^-\right)\cup U_{\rho_0,L}\cup
\left(\partial B^{\mathbb{R}^{n+1}}_{y_0}(1)\setminus B_{\rho_0}^+\right)
\ee
is a smooth hypersurface of $\mathbb{R}^{m+1}$. Notice that $\bar{g}$ is rotationally 
symmetric in the sense that the natural action of $SO(m)$ is isometric.
\end{rmrk}

\section{Revised Pipe Filling Technique in Positive Scalar Curvature}\label{sec:revisedpipe}
In this section, we state and prove our main technical result, Theorem \ref{T:pipefilling}.
In it, we clarify and expand upon the \emph{pipe filling technique} originally 
introduced by the third named author in the appendix to \cite{SW-JDG}. 
More details are given here that could be used to clarify the
construction in that appendix, as well. To write a detailed proof and to incorperate
the positive scalar curvature condition, we change the technique somewhat.  

Let us introduce the setting for the pipe filling technique.
For $m\geq3$, let $(N^m,g)$ be an oriented closed $m$-dimensional Riemannian manifold
with positive scalar curvature. Suppose we are given points $p,q\in N$ and a radius 
$\rho\leq\mathrm{inj}^N$ so that $B^N_p(\rho)$ and $B^N_q(\rho)$ are isometric to
geodesic balls in the unit sphere $\sphere^m$. We also fix a length $L>0$.
According to Theorem \ref{T:GLtunnel}, there is number $\rho_0\in(0,\rho)$
and a manifold $(U_{\rho_0,L},\bar g)$ which is topologically $\sphere^{m-1}\times [0,1]$,
has positive scalar curvature, and smoothly glues to $N$ to form
\be\label{eq:Nrho}
N_\rho:=N\setminus(B^N_p(\rho)\sqcup B^N_q(\rho))\cup U_{\rho_0,L}.
\ee
Denote the corrisponding integral current space by 
$\mathcal{N}_\rho=(N_\rho,d_{N_\rho},\int_\rho)$.

Following Remark \ref{rmk:gromovlawson}, there is a function 
$r_{\rho_0,L}:[\frac{-L'}{2},\frac{L'}{2}]\to[0,\rho]$ which, upon properly rotating 
its graph, produces the Gromov-Lawson tunnel $(U_{\rho_0,L},\bar g)$. We 
emphasize that the function $r_{\rho_0,L}$
depends only on $\rho_0,L$ and is independant of the ambient geometry of $(N,g)$.

We also let $N_0$ denote the space obtained by joining $p$ and $q$ by a line segment
\be\label{eq:Nzero}
N_0:=N\sqcup_{p\sim0,q\sim L}[0,L].
\ee
Denote the corrisponding integral current space by 
$\mathcal{N}_0=(N_0,d_{N_0},\int_{\sphere^m})$.

\bt\label{T:pipefilling} 
Let $(N^m,g)$ an oriented closed $m$-dimensional Riemannian manifold
with positive scalar curvature. Let $p,q\in N$ and
$\rho\in(0,\mathrm{inj}^N)$ be so that $B^N_\rho(p)$ and $B^N_\rho(q)$
are isometric to a spherical geodesic ball $B^\sphere_\rho$. Fix $L>0$ and
let $\mathcal{N}_0$ and $\mathcal{N}_\rho$
be the resulting integral current spaces described above in (\ref{eq:Nrho}) and (\ref{eq:Nzero}).

Then there is a constant $C>0$, continuously depending only on 
$L$, $\mathrm{Vol}(N)$, and $\mathrm{diam}(N)$, so that
\be\label{E:pipfillingIFestimate}
	d_\Fm\big(\mathcal{N}_0,\mathcal{N}_\rho\big)\leq C\cdot\rho
\ee
and
\be\label{E:pipfillingGHestimate} 
	d_{GH}\big((N_0, d_{N_0}), (N_\rho, d_{N_\rho})\big) \leq C\cdot\rho.
\ee
\et
The remainder of this section is devoted to the proof of Theorem \ref{T:pipefilling}.

\subsection{Constructing the common space}
We begin by constructing an integral current space 
$Z_\rho''$ with boundary $N_0\sqcup -N_\rho$ which we will use to
estimate the flat and Gromov-Hausdorff distances. To begin, consider the product
\[
Z_\rho:=\left(N\setminus (B^N_p(\rho)\sqcup B_q^N(\rho)\right)\times[-h_0,\rho+h]
\]
where $h_0$ and $h$ are given by
\be\label{choose-h}
h= \sqrt{2\rho\diam(N) - \rho^2\,}
\ee
\be\label{choose-h_0}
h_0 =\sqrt{2\pi\rho \diam(N) +8\rho}.
\ee
The next step is to produce a space $A_{\rho_0,L}$ which will glue into $Z_\rho$ along its boundary.

Following Remark \ref{rmk:gromovlawson}, the Gromov Lawson tunnel $U_{\rho_0,L}$
can be viewed as a hypersurface in $\mathbb{R}\times\mathbb{R}^n$ which 
transitions between two unit spheres equipped with spherical geodesic balls 
\[
B_{\rho_0}^-\subset \partial B^{\mathbb{R}^{m+1}}_{-y_0}(1), 
B_{\rho_0}^+\subset \partial B^{\mathbb{R}^{m+1}}_{y_0}(1)
\subset\mathbb{R}\times\mathbb{R}^m.
\]

The space $A_{\rho_0,L}$, pictured in Figure \ref{fig-tunnel-wall}, will be 
constructed in three portions 
\[
A_{\rho_0,L}=A_{1}\cup A_{2}\cup A_{3}.
\]
The first piece is a product of the Gromov-Lawson tunnel with the segment $[-h_0,0]$,
\begin{align}
A_{1}=\Big{\{}(t,x,s)\in[\frac{-L'}{2},\frac{L'}{2}]\times\mathbb{R}^{n}\times[-h_0,0]
	\colon |x|=r_{\rho_0,L}(t)\Big{\}}\notag.
\end{align}

The second piece consists of two parts: half of a rotated 
Gromov-Lawson tunnel we call the {\emph{pipe}}, $P$, and a cuspoidal region, 
$W$ as in Figure~\ref{fig-tunnel-wall}. More precisely, 
\be
A_{2}=P \cup W
\ee
where 
\begin{align}
P=\Big{\{}(t,x,s)\in\Big{[}\frac{-L'}{2},\frac{L'}{2}\Big{]}\times\mathbb{R}^{m}
	\times[0,\rho]&\colon |x|^2+s^2=r_{\rho_0,L}(t)^2\Big{\}}\notag\\
{}&-\left(B^{\mathbb{R}^{m+1}}_{-y_0}(1)\cup B^{\mathbb{R}^{m+1}}_{y_0}(1)\right)\times[0,\rho]\notag
\end{align}
and
\begin{align}
W=\left(B_{\rho_0}^-\cup B_{\rho_0}^+\right)\times[0,\rho]&{}\notag\\
-\{(t,x,s)\in&\Big{[}\frac{-L'}{2},\frac{L'}{2}\Big{]}\times\mathbb{R}^{m}\times[0,\rho+h]\colon
	|x|^2+s^2\leq r_{\rho_0,L}(t)^2\}.
\end{align}
Notice that the union $A_1\cup A_2$ is a $C^1$-hypersurface of 
$[\frac{-L'}{2},\frac{L'}{2}]\times\mathbb{R}^{n}\times[-h_0,\rho]$.

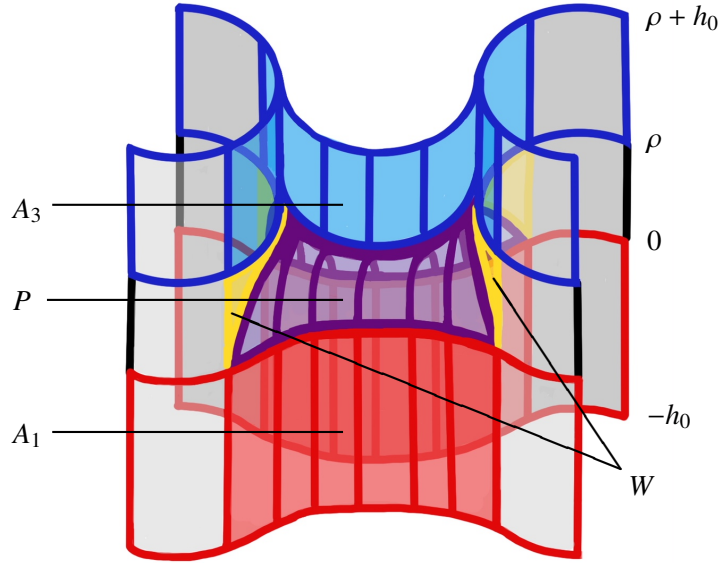
\begin{figure}[h!]
\includegraphics[scale=.4]{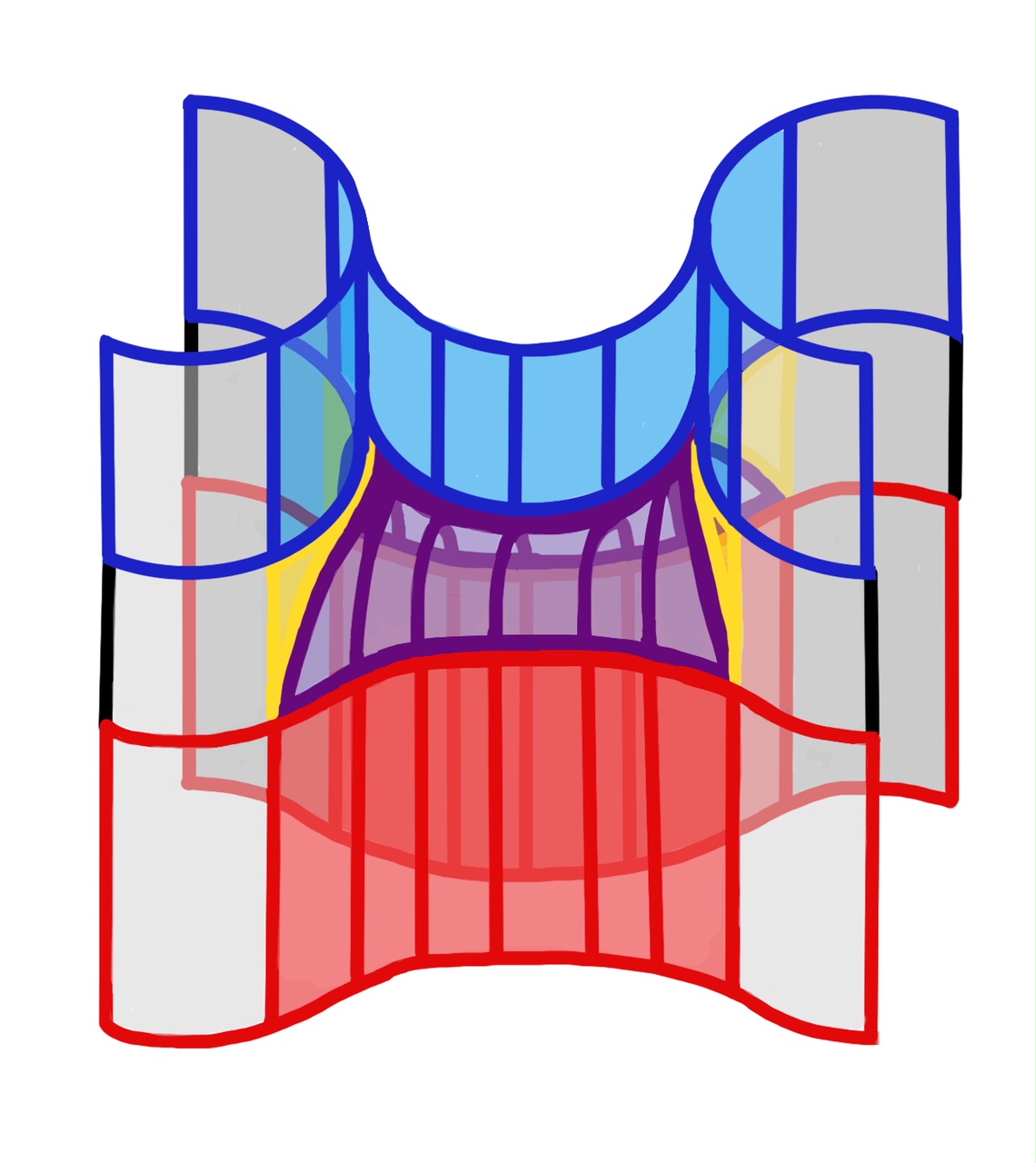}
\thicklines
\put(-250,70){$A_1$}
\put(-235,74){\line(1,0){110}}
\put(-250,120){$P$}
\put(-237,124){\line(1,0){112}}
\put(-250,155){$A_3$}
\put(-235,160){\line(1,0){110}}
\put(-17,50){$W$}
\put(-20,60){\line(-5,2){147}}
\put(-20,60){\line(-2,3){48}}
\put(-10,77){$-h_0$}
\put(-10,143){$0$}
\put(-10,182){$\rho$}
\put(-10,228){$\rho+h_0$}
\caption{Schematic depiction of the region $A_{\rho_0,L}$.}
\label{fig-tunnel-wall}
\end{figure}

The third and final piece, $A_3$, is given by
\be
A_{3}=\left(B_{\rho_0}^-\cup[0,L]\cup B_{\rho_0}^+\right)\times[\rho,\rho+h]
\ee
where the line segment $[0,L]$ is attached to the centers of $B_{\rho_0}^-$
and $B_{\rho_0}^+$. Notice that the subset 
\[
\left(B_{\rho_0}^-\cup[0,L]\cup B_{\rho_0}^+\right)\times\{\rho\}\subset A_3
\]
may be isometrically identified with the top portion of $A_2$ so that the 
union $A_{\rho_0,L}=A_1\cup A_2\cup A_3$ is a well-defined integral current space.

Finally, a portion of the boundary of $A_{\rho_0,L}$ can be isometrically identified with 
$(\partial B^N_{p}(\rho)\sqcup \partial B^N_q(\rho))\times[-h_0,\rho+h]$.
Gluing along this edge, we form the integral current space 
$Z''_\rho=Z_\rho\cup A_{\rho_0,L}$.

Notice that there is a (non-continuous) height function $H:Z''_\rho\to[-h_0,\rho+h]$
and that the hypersurfaces $H^{-1}\{-h_0\}$ and $H^{-1}\{\rho+h\}$ can be identified
with $-N_0$ and $N_\rho$, repsectively. Let $\phi_0:N_0\to Z''_\rho$
and $\phi_\rho:N_\rho\to Z''_\rho$ denote the corresponding inclusions. 
We also denote $Z_\rho':=H^{-1}([0,\rho+h])$.
Using the volume estimate
of Theorem \ref{T:GLtunnel} and Remark \ref{rmk:gromovlawson}, there is
a constant $C>0$, independant of $L$ and $\rho_0$, so that
\begin{equation}\label{eq:Pvol}
\vol(P)\leq CL\rho_0^m.
\end{equation}

\subsection{The embeddings $\varphi_0$ and $\varphi_\rho$ are distance preserving}
Here we prove the $N_0$ and $N_\rho$ embed
into $Z''_\rho$ by distance preserving maps, see Lemmas~\ref{L:pipefillisomemb1} 
and~\ref{L:pipefillisomemb2}.   
\bl\label{L:pipefillisomemb1}
The embedding $\varphi_0: N_0 \to Z''_\rho$
is distance preserving. 
\el
\begin{proof}
Assume on the contrary that there exists points $x_1,x_2 \in N_0$ such that
\be 
d_{Z''_\rho}(\varphi_0(x_1), \varphi_0(x_2))< d_{N_0}(x_1,x_2).
\ee  Then
there exists an arclength parametrized curve $\sigma:[0,D]\to Z''_\rho$
such that 
\be
L(\sigma)=D < d_{N_0}(x_1,x_2), \,\,\,\sigma(0)=x_1, \textrm{ and }\sigma(D)=x_2.
\ee
For ease of notation, we will use $\sigma$ to denote both the parameterization and its image.
The remainder of the proof will be broken into claims.
\begin{claim}\label{C:pipefillisomemb1-1}
$\sigma$ intersects the set
\be
Q:=\partial P\setminus H^{-1}\{0\}.
\ee
\end{claim}
\begin{proof}
Inspecting the definition of $Z''_\rho$, there is a distance nonincreasing map from
$H^{-1}([-h_0, 0])$ to $H^{-1}\{0\}$ and so we may assume that $\sigma$ lies in $Z'_\rho$.   

Let $\psi_0: Z'_\rho \setminus P \to \varphi_0(N_0)$ be the
continuous distance nonincreasing map which takes
\be \label{psi_0}
(x,t) \in Z'_\rho\setminus P
\ee
to $\psi_0(x,t)=(x,h+\rho)$. If $\sigma$ avoids the pipe, then
$\psi_0\circ \sigma$ is a curve running between
$\varphi_0(x_1)$ and $\varphi_0(x_2)$ whose length is no greater than $D$
and lies entirely in $\varphi_0(N_0)$, which is a contradiction. It follows that 
$\sigma\cap P\neq\emptyset$.

Now let $\psi_P: P \subset Z'_\rho \to \varphi_0(N_0)$ be 
the continuous distance nonincreasing map defined by sending
\be
(y,t) \in P\cong \sphere^{m}_+ \times [0,L]
\ee 
to a point in the image of the thread: 
\be
\psi_P(y,t)=(t,\rho+h) \in [0,L]\subset\varphi_0(N_0).
\ee

Let $\psi: Z'_\rho \to \varphi_0(N_0)$ be defined so
that $\psi(z)= \psi_0(z)$ for $z \in Z'_\rho \setminus P$ and
$\psi(z)=\psi_P(z)$ for $z\in P$.  This map is continuous everywhere
except on the set $Q$.
If our curve $\sigma$ avoids this set, then
$\psi\circ \sigma$ is a shorter curve between its endpoints  
whose image lies in $\varphi_0(N_0)$
which is a contradiction. This establishes Claim \ref{C:pipefillisomemb1-1}.
\end{proof}
\begin{claim}\label{C:pipefillisomemb1-2}
Let $x\in N$, $y\in \{p, q\}$, and $z \in Q_y$ where $Q_y$ is the component of 
$Q$ closest to $y$.
Then we have
\be
d_{Z'_\rho}(\varphi_0(x),z) \ge d_{Z'_\rho}(\varphi_0(x), (y, h+\rho)).
\ee
\end{claim}
\begin{proof}
Let $k=d_{Z'_\rho}(\varphi_0(x), (y, h+\rho))$. Notice that by our
choice of $h$ in (\ref{choose-h}) we have 
\be
h^2= 2\rho\diam(N) -\rho^2 \ge 
2 \rho k -\rho^2=k^2 -(k-\rho)^2.
\ee
Using the product structure of this portion in $Z'_\rho$, we obtain
\begin{eqnarray}
d_{Z'_\rho}(\varphi_0(x),z)^2 &\ge& d_{N}(\varphi_0(x),\psi(z))^2 + (\rho+h-H(z))^2\notag\\
&\ge & (d_N(x,y)-\rho)^2 + h^2 \notag\\
&=&(k-\rho)^2 + h^2 \ge k^2,
\end{eqnarray}
completing the proof of Claim \ref{C:pipefillisomemb1-2}.
\end{proof}

Our next task is to produce a competitor to $\sigma$.
According to Claim \ref{C:pipefillisomemb1-1} we may consider the first and 
final times -- denoted by $L_1$ and $L_2$, respectively -- the path 
$\sigma$ passes through the set $Q$. 
We replace the segment $\sigma([0,L_1])$
by a minimizing $N_0$-geodesic $\gamma_1:[0,1]\to \varphi_0(N_0)$ running 
directly from $C(0)$ to $\psi(\sigma(L_1))$.   
Likewise, we replace the segment $\sigma([L_2, D])$
by a minimizing $N_0$-geodesic $\gamma_2:[2,3]\to\varphi_0(N_0)$ 
running directly from $\psi(\sigma(L_2))$ to $\sigma(D)$. By Claim \ref{C:pipefillisomemb1-2}, 
we know that $L(\gamma_1)\leq L_1$ and $L(\gamma_2)\leq L_2$. 
 
Now we have two possibilities: either $\psi(\sigma(L_2))=\psi(\sigma(L_1))$ or 
$\psi(\sigma(L_2))\neq\psi(\sigma(L_1))$.
If $\psi(\sigma(L_2))=\psi(\sigma(L_1))$, then we may form
the concatination $\gamma=\gamma_1*\gamma_2$ which is
path from $\varphi_0(x_1)$ to $\varphi_0(x_2)$ 
lying entirely within $\varphi_0(N_0)$ and has length $L(\gamma)\leq D$.
This contradicts our assumption and we are done.
Now assume that $\psi(\sigma(L_2))\neq\psi(\sigma(L_1))$. We can consider 
a minimizing $N_0$-geodesic $\gamma_3:[1,2]\to Z''_\rho$ which runs from 
$\psi(\sigma(L_1))$ to
$\psi(\sigma(L_2))$. Notice that $L(\gamma_3)\leq L(\sigma[L_1,L_2])$, 
which can be seen using $\phi$.
It follows that the concatination $\gamma_1*\gamma_3*\gamma_2$
is a path from $\varphi_0(x_1)$ to $\varphi_0(x_2)$ 
lying entirely within $\varphi_0(N_0)$ and has length $L(\gamma)\leq D$, 
also yielding the desired contradiction.
\end{proof}

\bl\label{L:pipefillisomemb2}
The map $\varphi_\rho: N_\rho \to Z''_\rho$ is distance preserving.
\el
\begin{proof}
Assume on the contrary that there exists $x_1,x_2 \in N_\rho$ such that
\be
d_{Z''_\rho}(\varphi(x_1), \varphi(x_2))< d_{N_\rho}(x_1,x_2).
\ee  
Then
there exists an arclength parametrized curve $\sigma:[0,D]\to Z''_\rho$
such that $L(\sigma)=D < d_{N_\rho}(x_1,x_2)$ and 
$\sigma(0)=\varphi(x_1)$ and $\sigma(D)=\varphi(x_2)$. 
As before, we will break the remainder of the proof into smaller claims.
\begin{claim}\label{C:pipefillisomemb2-1}
We may assume that $\sigma\subset H^{-1}[-h_0,\rho)$ and $\sigma$ 
must intersect the set $H^{-1}(0)$.
\end{claim}
\begin{proof}
We first notice that we may assume $\sigma\subset H^{-1}[-h_0,\rho)$.   
This can be seen by applying the distance nonincreasing map 
$\psi_1: H^{-1}[\rho, h+\rho] \to H^{-1}(\rho)$ defined by
$\psi_1(x,t)=(x,\rho)$ to any portion of $\sigma$ lying above height $\rho$.   
On the other hand, there is a projection mapping $\psi_2:H^{-1}[-h_0,0]\to H^{-1}(-h_0)$
which is distance nonincreasing due to the product structure of $Z''_\rho$ there.
It follows that $\sigma$ must pass through the set $H^{-1}(0,\rho)$. 
Since $\sigma$ begins and ends at height $-h_0$, $\sigma$ must intersect $H^{-1}(0)$.
\end{proof}

\begin{claim}\label{C:pipefillisomemb2-2'}
The curve $\sigma$ can be replaced by a curve $\sigma':[0,D']\to Z''_\rho$ satisfying the following
\begin{enumerate}
\item $\sigma'(0)=\sigma(0)$ and $\sigma'(D')=\sigma(D)$;
\item $L(\sigma)\geq L(\sigma')-8\rho$;
\item $\sigma'$ does not pass through the interior of the region 
$A_2\subset H^{-1}([0,\rho])$;
\end{enumerate}
\end{claim}
\begin{proof}
If $\sigma$ passes through $W$, we replace this segment of $\sigma$ by a path lying entirely within
$\partial W$. Since the diameter of $\partial W$ is less than $4\rho+\pi\rho<8\rho$, this replacement yields a curve $\bar{\sigma}$ with $L(\sigma)\geq L(\bar{\sigma})-8\rho$.

Finally, suppose $\bar{\sigma}$ intersects the pipe $P$. By perhaps altering 
$\bar \sigma$, but preserving 
its length bound, we can assume that the first and final times $\bar \sigma$ passes through $P$
occur on $H^{-1}\{0\}\cap P$. Now due to the rotational symmetry of $P$,
 the distance between two points in the
equitorial slice $H^{-1}\{0\}$ can be realized by a path contained in it. It follows that 
$\bar \sigma$ may be replaced 
by a curve $\sigma'$ of no greater length which does not pass through the interior of $P$.
\end{proof}
Since $\sigma'$ lies in $H^{-1}[-h_0,\rho)$ and does not pass through the interior of $A_2$,
we may assume that $\sigma'$ lies in $H^{-1}[-h_0,0]$ due to 
the product structure of $H^{-1}[0,\rho)\setminus A_2$. Moreover,
we may assume that $\sigma'\cap H^{-1}\{0\}$ is connected by projecting any portion exiting and
entering $H^{-1}\{0\}$ back up to $H^{-1}\{0\}$.

We proceed in two cases. First suppose the original curve $\sigma$ 
does not pass through the region $A_2$. Since the region 
$H^{-1}[-h_0,\rho]\setminus A_2$ has a product structure,
we can project $\sigma$ down to $H^{-1}\{-h_0\}$ to obtain a shorter curve.

Now suppose $\sigma$ -- and hence $\sigma'$ -- intersects $A_2$. Consider times $L_1$ and $L_2$
which correspond to the first and final times, respectively, when $\sigma'$ passes through 
$H^{-1}(0)$. We will proceed by constructing a competing curve to $\sigma'$.
Let $\gamma_1:[0,1]\to \varphi_\rho(N_\rho)$ be a minimizing $N_\rho$-geodesic
traveling from $\varphi_\rho(x_1)$ to $\psi_2(\sigma'(L_1))$. Likewise, let 
$\gamma_2:[2,3]\to\varphi_\rho(N_\rho)$
be a minimizing $N_\rho$-geodesic traveling from $\psi_2(\sigma'(L_2))$ to $\varphi_\rho(x_2)$.
Using the product structure of $H^{-1}[-h_0,0]$, we immediately obtain the following claim.
\begin{claim}\label{C:pipefillisomemb2-2}
The following inequalities hold:
\begin{equation}\label{E:pipefillisomemb2-2}
L(\sigma'[0,L_1])\geq\sqrt{L(\gamma_1)^2+h_0^2},\quad\quad L(\sigma'[L_2,D])\geq\sqrt{L(\gamma_2)^2+h_0^2}.
\end{equation}
\end{claim}
Finally, let $\gamma_3:[1,2]\to H^{-1}(-h_0)$ be a minimizing $N_\rho$-geodesic 
running from $\psi_2(\sigma'(L_1))$ to $\psi_2(\sigma'(L_2))$.
Since $\sigma'[L_1,L_2]$ lies entirely in $H^{-1}\{0\}$, the following inequality holds:
\begin{equation}\label{E:pipefillisomemb2-3}
L(\sigma'[L_1,L_2])\geq L(\gamma_3).
\end{equation}

Now, by our choice of $h_0$ in (\ref{choose-h_0}), one can show
\begin{align}
h_0^2&\geq2\mathrm{diam}(N)\rho\pi+\pi^2\rho^2\notag\\
{}&\geq2L(\gamma_1)\rho\pi+\pi^2\rho^2\notag\\
{}&=(L(\gamma_1)+\pi\rho)^2-L(\gamma_1)^2\notag
\end{align}
and it follows that 
\begin{equation}\label{E:pipefillisomemb2-4}
\sqrt{L(\gamma_i)^2+h_0^2}\geq L(\gamma_i)+\rho\pi
\end{equation}
holds for $i=1,2$.
To conclude, we may sum inequalities \ref{E:pipefillisomemb2-2} and \ref{E:pipefillisomemb2-3},
to find
\begin{align}\label{E:pipefillisomemb2-5}
L(\sigma)&\geq L(\sigma')-8\rho\notag\\
{}&\geq \sqrt{L(\gamma_1)^2+h_0^2}+\sqrt{L(\gamma_2)^2+h_0^2}+
	L(\gamma_3)-8\rho\notag\\
{}&\geq L(\gamma_1)+L(\gamma_2)+L(\gamma_3)\notag.
\end{align}
This shows that the concatination $\gamma=\gamma_1*\gamma_3*\gamma_2$ is a 
competitor to $\sigma$ of no greater length 
which lies entirely in the embedding $\varphi_\rho(N_\rho)$, yielding a contradiction.
\end{proof}

\subsection{Proving the pipe filling theorem}
Now that we have confirmed that the integral current spaces $\mathcal{N}_0$ and 
$\mathcal{N}_\rho$ have distance-preseving embeddings into the boundary of $Z''_\rho$,
the proof of Theorem \ref{T:pipefilling} will be complete upon estimating the volume
of $Z''_\rho$.
\begin{proof}[Proof of Theorem~\ref{T:pipefilling}.]
Let $B_\rho$ denote the current on $Z''_\rho$ defined by integration on its
$m$-dimensional stratum.
According to Lemmas~\ref{L:pipefillisomemb1} and~\ref{L:pipefillisomemb2},
the integral current space $(\bar{Z}, d_{Z''_\rho}|_{\bar{Z}\times\bar{Z}},B_\rho)$
with embeddings $\varphi_0$ and $\varphi_\rho$ is an admissable set of data
with which to estimate the intrinsic flat distance between $\mathcal{N}_0$ and $\mathcal{N}_\rho$.
In other words,
\be
	d_{\mathcal{F}}(\mathcal{N}_0,\mathcal{N}_\rho)\leq\mass(B_\rho) = \vol(\bar{Z}).
\ee

We will estimate the volume of $\bar{Z}$ by inspecting the three peices in its construction
\[
\vol(\bar{Z})=\vol(N_\rho \times [-h_0,0]) + \vol(Z_\rho) + \vol(Z'_\rho).
\]

By the volume estimate in Theorem \ref{T:GLtunnel}, we can estimate
\begin{align*}
	\vol(N_\rho \times [-h_0,0]) &= h_0 \cdot \vol(N_\rho)  \\
		&= h_0 \left( \vol(N) - 2\vol(B_p(\rho)) +\vol(P\cap H^{-1}(0))) \right) \\
		&\le h_0 \left( \vol(N) + C_1L \rho^m \right)
\end{align*}
and
\begin{align*}
	\vol(Z_\rho) &\leq \vol(N \times [0,\rho]) + \vol(P) \\
		&\le \rho \vol(N) +C_2L\rho^m
\end{align*}
where $C_1,C_2>0$ come from Theorem \ref{T:GLtunnel} and estimate (\ref{eq:Pvol}) and are 
independant of $\rho$ and $L$.
The final piece of $\bar{Z}$ evidently has volume 
\[
\vol(Z'_\rho) \leq \vol(N \times [\rho,\rho+h]) = h \vol(N).
\]
Summing the above three inequalities, we obtain the estimate \ref{E:pipfillingIFestimate}.

Finally, we consider the Gromov-Hausdorff distance between $N_\rho$ and $N_0$. 
Since the maps $\varphi_\rho: N_\rho \to Z''_\rho$ and $\varphi_0: N_0 \to Z''_\rho$ are 
distance-preserving,
\be
d_{GH}(N_\rho,N_0) \le d_H^{Z''_\rho}(\varphi_\rho(N_\rho),\varphi_0(N_0)).
\ee	
By considering the lengths of verticle paths running between $H^{-1}(\rho+h)$
and $H^{-1}(-h_0)$, it follows that 
\[
 d_H^{Z''_\rho}(\varphi_\rho(N_\rho),\varphi_0(N_0))\leq h_0+2\pi\rho+h,
\]
where we have used part $(5)$ of Theorem \ref{T:GLtunnel}. The 
inequality (\ref{E:pipfillingGHestimate}) follows.
\end{proof}

\section{Constructing spheres with tunnels and spheres with threads}\label{sec:construction}
In this section we construct the sequences we will use to prove Theorem~\ref{T:mainresult}.
First we construct spheres with
increasingly dense threads, denoted by $(Y_{\varepsilon_j}, d_{Y_{\varepsilon_j}})$,
and spheres with increasingly dense tunnels, denoted by 
$(X_{\varepsilon_j}, d_{X_{\varepsilon_j}})$, see Definitions~\ref{D:defofXepsilon} and 
\ref{D:defofYepsilon}, respectively.   
To $(Y_{\varepsilon_j}, d^{Y_{\varepsilon_j}})$ we associate the integral current space
$(\sphere^m, d_{Y_{\varepsilon_j}}, \int_{\sphere^m})$, 
see Definition~\ref{D:ICSassociatedtoYep},
which no longer contains the threads. The second sequence 
$(X_{\varepsilon_j}, d_{X_{\varepsilon_j}})$ also has
an integral current structure, $T_{\vare_j}$, described in Definition~\ref{D:ICSassociatedtoXep}.
\subsection{The geodesic spaces $(Y_\varepsilon, d_{Y_\varepsilon})$ 
	and $(X_\varepsilon, d_{X_\varepsilon})$}
In this subsection, we fix $\varepsilon>0$ and construct two geodesic spaces: a sphere
with threads $(Y_\varepsilon, d_{Y_\varepsilon})$ and a sphere with tunnels 
$(X_\varepsilon, d_{X_\varepsilon})$. 

We begin by choosing a collection of points $\{p_1,\ldots, p_{N(\varepsilon)}\}$ 
in the sphere $\sphere^m$ so that 
$\{B_{p_i}(\varepsilon)\}_{i=1}^{N(\varepsilon)}$ are pairwise disjoint and
$\{B_{p_i}(2\varepsilon)\}_{i=1}^{N(\varepsilon)}$ 
forms an open cover of $\sphere^m$. The number of points
required to form such a collection, $N(\varepsilon)$, is
on the order of $\sin(\varepsilon)^{-m}$, though we will not need explicit knowledge of it.

Next, for every $\varepsilon>0$ and $i\in\{1,2,\ldots N(\varepsilon)\}$, we will 
choose $N(\varepsilon)-1$ points
which lie on the geodesic sphere $\partial B_{p_i}(\varepsilon)$.
We denote the points by $\{q_j^i\}_{j\in\{1,\ldots,\hat{i},\ldots N(\varepsilon)\}}$
where $\hat{i}$ indicates that the index $i$ is ommitted.
We choose $q_j^i$ sufficiently spaced so that 
\be\label{E:qspacing}
d_{\sphere^m}(q_j^i,q_{j'}^i)>\frac{\varepsilon}{N(\varepsilon)}
\ee
holds for each $j\neq j'$ and all $i$.

We now describe a way to pairing the points $q_j^i$ between different balls
$B_{p_i}(\varepsilon)$ so that we may connect them with line segments or tunnels. For each 
$i \in \{1,\ldots,N(\varepsilon)\}$ and $j \in\{1,\ldots,\hat{i},\ldots,N(\varepsilon)\}$, we pair 
$q_j^i\in\partial B_{p_i}(\varepsilon)$ with $q_i^j\in\partial B_{p_j}(\varepsilon)$.
Let $L_j^i$ be the Euclidean distance between paired points
\be
L_j^i := d_{\E^{m+1}}(q_j^i,q_i^j)=\sqrt{2 -2 \cos\left( d_{\sphere^m}(q_j^i, q_i^j) \right)}.
\ee
Now we may define the {\emph{sphere with threads}}.
\bd\label{D:defofYepsilon} 
Define $Y_\varepsilon$ to be the sphere with attached line segments $[0, L_j^i]$ 
by identifying its boundary to the paired points $\{q_j^i,q_i^j\}$,
\begin{align}\label{E:defofYepsilon}
	Y_\varepsilon &= \sphere^m \bigcup_{i<j}^{N(\varepsilon)} [0,L_j^i]. 
\end{align}
We give $Y_\varepsilon$ the induced length metric, denoted by $d_{Y_\varepsilon}$.
\ed
Next, we construct the {\emph{sphere with tunnels}}.
\bd\label{D:defofXepsilon}
Fix the radius
\[
\rho(\varepsilon)=\frac{\varepsilon}{N(\varepsilon)^{2}}
\]
and notice that inequality (\ref{E:qspacing}) implies the geodesic balls 
$\{B_{q_j^i}(\rho(\varepsilon))\}$ are disjoint.
Define $X_\varepsilon$ by removing the balls $B_{q_j^i}(\rho(\varepsilon))$, 
for each $i\in\{1,\ldots,N(\varepsilon)\}$ 
and $j\in\{1,\ldots,\hat{i},\ldots,N(\varepsilon)\}$, and attaching 
Gromov-Lawson tunnels $U_i^j=S^{m-1}\times[0,1]$ of length $L_j^i$ and radius 
$\rho(\varepsilon)$ along the boundaries of $B_{q_j^i}(\rho(\varepsilon))$
and $B_{q_i^j}(\rho(\varepsilon))$. 

More precisely, we equip $U_i^j$ with the
Gromov-Lawson tunnel obtained by the construction in Section \ref{sec:revisedpipe}.
This yields the Riemannian manifold of positive scalar curvature
\begin{align}\label{E:defofXepsilon}
	X_\varepsilon &:= \left( \sphere^m \setminus \bigcup_{i\neq j}^{N(\varepsilon)}B_{q^i_j}(\rho(\varepsilon)) \right)
		 \bigcup_{i<j}^{N(\varepsilon)} U_i^j. \notag
\end{align}
We denote the induced length metric by $d_{X_\varepsilon}$. 
\ed
\subsection{The integral current spaces $N_\vare$ and $M_\vare$ are defined:} 
We now describe the natural integral current space structures associated to these metric spaces.
 Recall that definition of an integral current space requires that every point have positive 
$m$-dimensional density. 
\bd\label{D:ICSassociatedtoYep}
The integral current we associate to $(Y_\varepsilon,d_{Y_\varepsilon})$ is integration over 
$\sphere^m$ and the integral current space is 
$(\sphere^m, d_{Y_\varepsilon}, \int_{\sphere^m})$, and we will be denoted by 
$N_\varepsilon$.
\ed
Note that the threads are removed in the integral current space and we are left with the sphere, 
but distances on the sphere are measured using $d_{Y_\varepsilon}$ instead of  $d_{\sphere^m}$. 
\bd\label{D:ICSassociatedtoXep}
The integral current associated to $(X_\varepsilon, d_{X_\varepsilon})$ is given by integration 
over the oriented Riemannian manifold $X_\varepsilon$. The integral current space will 
be denoted by $M_\varepsilon$. So,
\be
T_\varepsilon = T_1^\varepsilon + \sum_{i,j} T_{i,j}^\varepsilon,
\ee
where $T_1^\varepsilon$ is integration over the punctured sphere 
$\sphere^m \setminus U_\vare$, and $T_{i,j}^\varepsilon$ is integration over 
the Gromov-Lawson tunnel $U_i^j$. 
\ed
\section{Intrinsic Flat convergence of $M_{\vare_j}$}
We now prove our main result -- that $M_{\varepsilon_j}$ 
$\mathcal{F}$-converges to $M_\infty$ as $\varepsilon_j \to 0$. 
We first prove they are close to the $N_{\vare_j}$ 
[Proposition~\ref{P:Fclosetospherewthreads}] and then
prove the $N_{\vare_j}$ converge to $M_\infty$ [Proposition~\ref{P:Fclosetospherewrestricted}].   We end this section with
a proof of Theorem~\ref{T:mainresult}.
\subsection{Spheres with tunnels are close to spheres with threads}
We now apply Theorem~\ref{T:pipefilling} inductively to prove that
our sequence of spheres with increasingly dense tunnels 
(Definition~\ref{D:defofXepsilon}) are increasingly close to
our sequence of spheres with increasingly dense threads (Definition~\ref{D:defofYepsilon}).
\bp\label{P:Fclosetospherewthreads}
There is a constant $C>0$ so that, for all sufficiently small $\varepsilon>0$, we have
\be\label{E:spherewtunnelsIFclosetospherewthreads}
	d_\Fm(M_\varepsilon,N_\varepsilon) \le C\varepsilon.
\ee
\ep
\begin{proof}
For each $\varepsilon>0$, enumerate the collection of pairs 
$\{q_j^i,q_i^j\}\subset \sphere^m$ from $1$ to 
$K:=N(\varepsilon)(N(\varepsilon)-1)/2$. For $k\in\{1,\ldots, K\}$,
let $M_\varepsilon(k)$ denote the integral current space resulting from 
replacing the first $k$ tunnels $U_i^j$ in the construction of $M_\varepsilon$ 
with the corresponding 
threads $[0,L_j^i]$ as in the construction of $N_\varepsilon$. 
Notice that $M_\varepsilon(0)=M_\varepsilon$ and $M_\varepsilon(K)=N_\varepsilon$. 

We will require some more properties of $M_\varepsilon(k)$. Evidently, $|L_j^i|\leq 2$ 
and $\mathrm{diam}(M_\varepsilon(k))\leq \pi$ for all $k$ and all $\varepsilon$. 
Now part $(4)$ of Theorem \ref{T:GLtunnel} implies that there is
a constant $C_1>0$ so that $\vol(U_j^i)\leq C_1\rho(\varepsilon)^{m-1}$.
Using this and our choice of $\rho(\varepsilon)$, one may estimate
\begin{align}
\vol(M_\varepsilon(k))&\leq\vol(\sphere^m)+\sum_{i<j}^{N(\varepsilon)}\vol(U_j^i)\notag\\
{}&\leq\vol(\sphere^m)+N(\varepsilon)^2C_1\rho(\varepsilon)^{m-1}\notag\\
{}&=\vol(\sphere^m)+\varepsilon^{m-1} C_2\notag,
\end{align}
which is bounded above uniformly in $k$ and $\varepsilon$.
Having uniform control on the above quantities, we can conclude that
the constant in inequality (\ref{E:pipfillingIFestimate}) obtained from 
applying Theorem \ref{T:pipefilling} to a ball $B_{q_j^i}(\rho(\varepsilon))$ in 
$M_\varepsilon(k)$ is uniformly bounded in $k$ and $\varepsilon$ by some constant $C_3$.

Now we iteratively apply Theorem \ref{T:pipefilling}
to finish the proof:
\begin{align}
d_\Fm(M_\varepsilon,N_\varepsilon)&\leq\sum_{k=1}^Kd_\Fm(M_\varepsilon(k-1),
	M_\varepsilon(k))\notag\\
{}&\leq N(\varepsilon)^2 C_3\rho(\varepsilon)\notag\\
{}&=\varepsilon C_3\notag.
\end{align}
\end{proof}
\subsection{Spheres with threads are close to sphere with the restricted Euclidean metric:}
Our next goal is to show that the spheres with increasingly dense threads $N_\varepsilon$ 
converge to the sphere with the restricted Euclidean distance $M_\infty$ in the intrinsic flat sense.
\bp\label{P:Fclosetospherewrestricted}
The integral current spaces $N_\varepsilon$ converge to $M_\infty$
in the intrinsic flat sense as $\varepsilon\to0$.
\ep
\begin{proof}
Notice that the underlying spaces and integral currents of $Y_\varepsilon$ and 
$M_\infty$ are identical for all $\varepsilon$. In this setting, 
to show $N_\varepsilon\to M_\infty$ in the intrinsic flat sense, it suffices to show
\be\label{E:uniformconvergence}
\sup_{\sphere^m\times\sphere^m}|d_{Y_\varepsilon}-d_\mathbb{E}|\to0
\ee
and the existance of a $\lambda>0$, independant of $\varepsilon$, so that
\be\label{E:lipconvergence}
\frac{1}{\lambda}\leq\frac{d_{Y_\varepsilon}(x,y)}{d_\mathbb{E}(x,y)}\leq\lambda
\ee
for all $x,y\in\sphere^m$.
See \cite[Theorem 9.5.3]{SormaniChapter} for details. We will proceed by
verrifying conditions (\ref{E:uniformconvergence}) and (\ref{E:lipconvergence}).

For $\varepsilon>0$ and $x,y\in\sphere^m$, choose points $p_i$ and $p_{i'}$ in the net 
$\{p_i\}_{i=1}^{N(\varepsilon)}$ closest to $x$ and $y$, respectively. Now,
\begin{align}\label{E:Fclosetospherewrestricted1}
d_{Y_\varepsilon}(x,y)&\leq d_{\sphere^m}(x,q_{i'}^i)+d_{\mathbb{E}}(q_{i'}^i,q_i^{i'})
	+d_{\sphere^m}(q_i^{i'},y)\notag\\
{}&\leq6\varepsilon+d_{\mathbb{E}}(q_{i'}^i,q_i^{i'})
\end{align}
where the first inequality follows from the fact that 
$d_{\mathbb{E}}(q_{i'}^i,q_i^{i'})=d_{Y_\varepsilon}(q_{i'}^i,q_i^{i'})$
and the second inequality follows from our choice of points $q_i^j\in\partial B_{p_i}(\varepsilon)$.
Now we would like to compare $d_{\mathbb{E}}(x,y)$ and $d_{\mathbb{E}}(q_{i'}^i,q_i^{i'})$.
Notice that
\begin{align}\label{E:Fclosetospherewrestricted2}
d_{\mathbb{E}}(q_{i'}^i,q_i^{i'})&\leq d_{\mathbb{E}}(x,q_{i'}^i)+d_{\mathbb{E}}(x,y)+d_{\mathbb{E}}(q_i^{i'},y)\notag\\
{}&\leq 6\varepsilon+d_\mathbb{E}(q_{i'}^i,q_i^{i'}).
\end{align}
Combining inequalities (\ref{E:Fclosetospherewrestricted1}) and 
(\ref{E:Fclosetospherewrestricted2}), we obtain
\[
d_{Y_\varepsilon}(x,y)-d_{\mathbb{E}}(x,y)\leq12\varepsilon.
\]
This implies condition (\ref{E:uniformconvergence}).

All that remains is to verrify condition (\ref{E:lipconvergence}). 
To this end, we restrict our attention to $\varepsilon>0$ small enough so that 
\[
\frac{1}{100}r^2\leq 2-2\cos(r)
\]
holds for all $r\in[0,\rho(\varepsilon)]$. Now we consider two cases.

First, assume that $x,y\in\sphere^m$ satisfy $d_\mathbb{E}(x,y)\geq \rho(\varepsilon)$.
Then inequalities (\ref{E:Fclosetospherewrestricted1}) and 
(\ref{E:Fclosetospherewrestricted2}) imply 
\begin{align}\label{E:lipcase1}
\frac{d_{Y_\varepsilon}(x,y)}{d_\mathbb{E}(x,y)}\leq 13.
\end{align}
Next, assume that $x,y\in\sphere^m$ satisfy $d_\mathbb{E}(x,y)< \rho(\varepsilon)$.
In this case, notice that we may assume that $d_{Y_\varepsilon}(x,y)=d_{\sphere}(x,y)$ since
$B_x(\rho(\varepsilon))$ will intersect at most one of the balls 
$B_{q_j^i}(\rho(\varepsilon))$ for small enough $\varepsilon$. 
With this in mind, we can estimate
\begin{align}
d_\mathbb{E}(x,y)^2&=2-2\cos(d_{Y_\varepsilon}(x,y))\notag\\
{}&\geq\frac{1}{100}d_{Y_\varepsilon}(x,y)^2\notag.
\end{align}
Rearranging this, we find
\be\label{E:lipcase2}
\frac{d_{Y_\varepsilon}(x,y)}{d_\mathbb{E}(x,y)}\leq10.
\ee
Finally, notice that, for any $x,y\in\sphere^m$,
\be\label{E:lipcase3}
1\leq \frac{d_{Y_\varepsilon}(x,y)}{d_\mathbb{E}(x,y)}.
\ee
Combining (\ref{E:lipcase1}), (\ref{E:lipcase2}), and (\ref{E:lipcase3}),
we conclude that condition (\ref{E:lipconvergence}) holds with $\lambda=13$, finishing
the proof.
\end{proof}
\subsection{Proof of Theorem~\ref{T:mainresult}:}
By the triangle inequality, Propositons \ref{P:Fclosetospherewthreads}, 
and \ref{P:Fclosetospherewrestricted} we see that 
\[
	d_\Fm(M_\varepsilon,M_\infty) 
		\le d_\Fm(M_\varepsilon,N_\varepsilon) + d_\Fm(N_\varepsilon,M_\infty) 
		\to 0
\]
as $\vare \to 0$. Theorem \ref{T:mainresult} follows.

\bibliography{sphere1}
\bibliographystyle{amsalpha}

\end{document}